\documentclass[conference,letterpaper]{IEEEtran}
\IEEEoverridecommandlockouts
\usepackage[utf8]{inputenc}
\usepackage[english]{babel}

\usepackage{cite}
\usepackage{newclude}
\usepackage{amsmath,amssymb,amsfonts,amsthm}
\usepackage{algpseudocode}
\usepackage{algorithm}
\usepackage{graphicx}
\usepackage{textcomp}
\usepackage{xcolor}
\usepackage{multirow}
\usepackage{geometry}
\newtheorem{theorem}{Theorem}

\newtheorem{assumption}{Assumption}

\newcommand{\hwa}[1]{\textcolor{black}{#1}}

\def\BibTeX{{\rm B\kern-.05em{\sc i\kern-.025em b}\kern-.08em
    T\kern-.1667em\lower.7ex\hbox{E}\kern-.125emX}}
\begin{document}

\title{\vspace{0.25in}Distributed Edge Computing Task Allocation with Network Effects
\\
}

\author{\IEEEauthorblockN{Henry W. Abrahamson}
\IEEEauthorblockA{\textit{Department of ECE} \\
\textit{Northwestern University}\\
Evanston, IL, USA \\
henryabrahamson2022@u.northwestern.edu}
\and
\IEEEauthorblockN{Yongho Kim, Seongha Park}
\IEEEauthorblockA{\textit{Mathematics and Computer Science Division} \\
\textit{Argonne National Laboratory}\\
Lemont, IL, USA \\
\{yongho.kim, seongha.park\}@anl.gov} \\
\and
\IEEEauthorblockN{Ermin Wei}
\IEEEauthorblockA{\textit{Departments of ECE and IEMS} \\
\textit{Northwestern University}\\
Evanston, IL, USA \\
ermin.wei@northwestern.edu}
}

\maketitle

\begin{abstract}

Field-deployable edge computing nodes form a network and are used to complete scientific tasks for remote sensing and monitoring. The networked nodes collectively decide which scientific applications to run while they are constrained by various factors, such as differing hardware constraints from heterogeneous nodes and time-varying quality of service (QoS) requirements. We model the problem of task allocation as an optimization problem that maximizes the QoS, subject to the constraints. We solve the optimization problem using a dual-descent method, which can be easily implemented in a distributed way subject to the communication constraints of the network. Using a simulation that uses real-world data collected from Sage, a distributed sensor network, we analyze our policy's performance in dynamic situations where the required QoS and the nodes' capabilities change, and verify that it can adapt and return a feasible solution while accounting for those changes.
\end{abstract}



\section*{Acknowledgement}

This work was supported by the National Science Foundation (NSF) under
Grants CNS-2030251, and CMMI-2024774. This work utilized data from SAGE: A Software-Defined Sensor Network (NSF OAC-1935984), funded by the U.S. National Science Foundation’s Mid-Scale Research Infrastructure program. The research was supported by the U.S. Department of Energy under contract DE-AC02-06CH11357 and partially funded by the DOE Advanced Scientific Computing Research program through the grant Wildebeest: Migratory Computation for the Wireless 5G Digital Continuum (KJ0402020).

The submitted manuscript has been created by UChicago Argonne, LLC, Operator of Argonne National Laboratory (``Argonne"). Argonne, a U.S. Department of Energy Office of Science laboratory, is operated under Contract No. DE-AC02-06CH11357. The U.S. Government retains for itself, and others acting on its behalf, a paid-up nonexclusive, irrevocable worldwide license in said article to reproduce, prepare derivative works, distribute copies to the public, and perform publicly and display publicly, by or on behalf of the Government. The Department of Energy will provide public access to these results of federally sponsored research in accordance with the DOE Public Access Plan (http://energy.gov/downloads/doe-public-accessplan).

\section{Introduction}

In recent years, edge computing has brought new opportunities to move computation away from the cloud and to the edge, accelerating reaction times to local changes for time-sensitive applications. For example, in the field of environmental sensing, early wildfire detection via AI image processing to sense smoke is essential for first responders to plan and react before the wildfire becomes uncontrollable~\cite{altintas2015towards}. In this case, if a node detects a possible fire, it should report the detection to its neighbors so that they can adjust their own sensing and computing resource allocations and devote more to verify the existence of the fire or monitor its spread.


Task allocation in such urgent scenarios can be challenging because of the difficulties of dynamic task re-allocation, particularly when nodes' availability and maximum capabilities are also changing. Some tasks may need to be offloaded to other in-network nodes while one node is busy handling an urgent task. Other non-urgent tasks running on the same node may reduce their execution frequency to free up resources, impacting their data production rate. Some nodes may be lost due to the wildfire, or they may have inconsistent or unavailable communication capabilities. As such, the users who submitted those non-urgent tasks need to know how the execution of their tasks may be affected by these dynamic situations.

Determining a policy for task allocation for distributed edge nodes must consider various constraints of the problem, including both practical limitations from the hardware constraints of the system and the need to avoid undesirable behaviors. Network constraints also play a particularly significant role in developing schedulers for offloading tasks in the network~\cite{caminero2021quality, shreshta2021int, gao2024optimal}. In urgent computing~\cite{dazzi2024urgent}, however, the quality of service (QoS) becomes an important consideration because user applications capturing the urgent phenomenon should ideally be guaranteed for execution and may raise their execution frequency for monitoring the environmental conditions that change dynamically. 

Real-time queueing theory (RTQT) is one approach used to model urgent task allocation. Originally proposed in \cite{RTQTmain}, it provides a framework for analyzing queues of jobs that have deadlines, after which they expire and leave the queue. RTQT is used in applications such as video streaming \cite{RTQTvideo} and organ allocation for medical procedures \cite{RTQTorgan}. Much of the results in RTQT is primarily focused around maximizing throughput or minimizing missed deadlines, both common measures of QoS. However, for a multi-purpose edge computing system like Sage~\cite{Beckman:2019} that is responsible for multiple different tasks, a more flexible measure of QoS is needed. Different tasks might have different priorities that change over time, and different nodes in the network might complete a task with higher or lower quality outputs.

In situations like that, it becomes more advantageous to model task allocation in edge computing as a distributed optimization problem that can be tuned to induce the desired behavior for a specific application. Distributed optimization is a rich field with many different algorithms that offer different tradeoffs between computational complexity and convergence speed, such as EXTRA \cite{EXTRA}, DIGing \cite{DIGing}, accelerated dual-descent \cite{acceldualdescent}, and ADMM \cite{ADMM}. Due to its flexibility, distributed optimization has very broad applications including distributed games \cite{Yang2010optimgames}, distributed control \cite{nedic2018optimcontrol}, and data-processing for sensor networks \cite{rabbat2004optimsensors}. 

\hwa{
At this time however, Sage only relies on single-node decision-making, either running a job at a static rate or with a simple binary trigger that relies on purely local information, with no coordination between nodes \cite{sage_triggers}. 
As such, both casting our task allocation problem as a distributed optimization problem, and working towards an implementation of a simple distributed optimization algorithm to solve it within the constraints of the sensor network are necessary steps towards a full deployment of a distributed, reactive sensor network. Distributed optimization algorithms are well suited for this framework, as they utilize local communication and coordination instead of relying on potentially costly long-range or multi-hop communications to a central decision-maker \cite{Bertsekas_Tsitsiklis_1997}, ensuring a timely response to sudden, emergent events such as the previously mentioned wildfire example. }

Therefore, in this paper, we model the edge computing system and its given tasks as a distributed optimization problem that is trying to maximize some general QoS measure given the hardware and software constraints of the network. We pay particular attention to the user-derived constraints imposed by each different task, such as constraints on the placement of the tasks in the network. 

Our main contributions are: (1) providing a new formulation for task allocation in a distributed system by using an optimization framework, (2) showing that we can apply a standard distributed optimization algorithm to solve that problem, and (3) conducting simulations using real world data from Sage to show the benefits and efficacy of a distributed scheduler that can react quickly to changing situations in the environment as opposed to a centralized one that cannot.




\section{Formulation}
\label{sec:formulation}

In this section, we present our model of a network of distributed edge computing nodes, starting with an example to justify our assumptions. 

\subsection{Motivating Example: Wildfire Detection and Monitoring}
\label{sec:theexample}   

\begin{figure}
    \centering
    \includegraphics[width=\linewidth]{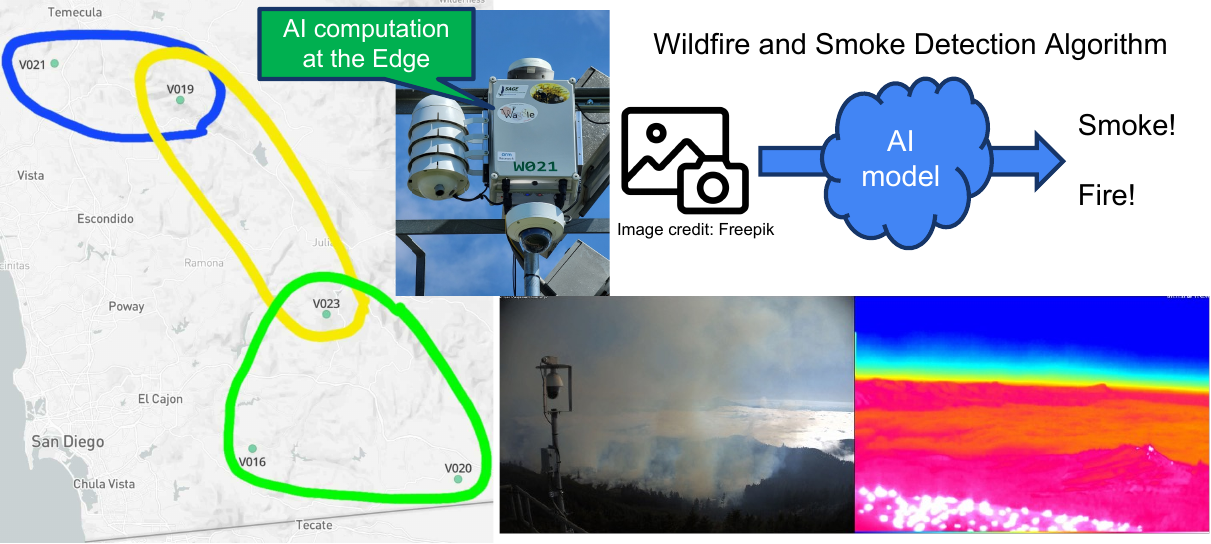}
    \caption{Wildfire and smoke detection as a use case: geographically grouped nodes run the smokey AI model~\cite{dewangan2022figlib} on visual and infrared camera images for early wildfire detection.}
    \label{fig:wildfire}
\end{figure}

Suppose that we have a set of 5 sensor nodes spread out across southern California, as in Fig. \ref{fig:wildfire}. These nodes are equipped with a camera to gather environmental data. One of their tasks, task 1, is to use image recognition for early detection of smoke that could lead to wildfires. However, the nodes are placed close together enough that measurements at some nodes will be highly correlated, if not identical to measurements at others. The nodes are therefore categorized by geographical placement and organized into 3 \textit{node groups}, represented by the colored regions. Their goal is to maximize the amount of measurements in each region, regardless of which exact node performs the task. 

We model this network as a graph in Fig. \ref{fig:testnetwork}. The node groups on the left correspond to the groups in Fig. \ref{fig:wildfire}. Recalling that the groups correspond to geographical proximity, it is natural that nodes within the same group be able to communicate with each other, so this system satisfies our communication requirement.

However, the nodes also have another task, task 2, from a different user from the multitenant system with a different scientific application.
Task 2 uses the same computing nodes to estimate the size of clouds in the sky for weather prediction and solar irradiance estimation. This task groups the nodes differently from the wildfire application, by decoupling the nodes 2 and 4. This might be the case because geographical proximity still determines general correlations, but there is some local circumstance that prevents nodes 2 and 4 from cooperating on task 2.
The network must therefore balance resources between wildfire detection and cloud analysis, being mindful that each task has different node groups for cooperation. 

\begin{figure}
    \centering
    \includegraphics[width=\linewidth]{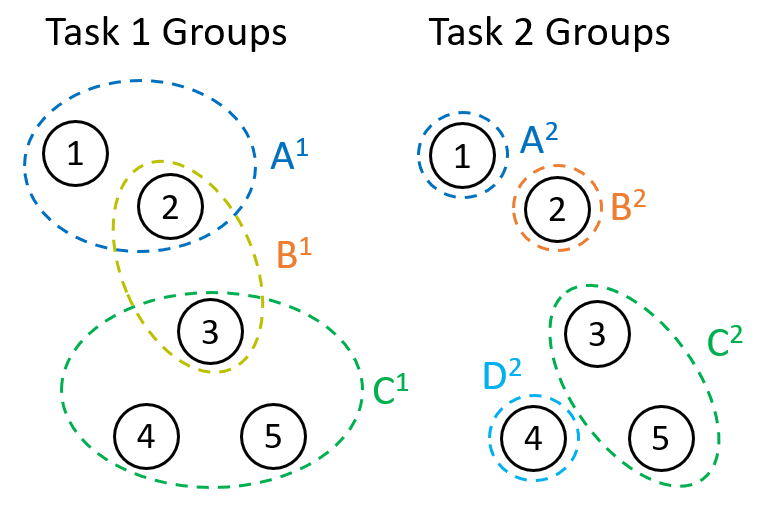}
    \caption{An example network. Each different colored dotted circle represents a different node group, with different node groups for task 1 (left) versus task 2 (right). Nodes within a group can communicate freely with each other. 
    }
    \label{fig:testnetwork}
\end{figure}

\subsection{Problem Formulation}
Let $\mathcal{N}$ be the set of nodes, and let $N$ be the total number of nodes in the network. Every node belongs to one or more node groups which each represents a set of nodes that are in some way able to meaningfully cooperate on a given task. 
Let $\mathcal{M}$ be the set of node groups, each associated with a task, and let $M$ be the total number of node groups. We represent each node group individually $m$ via the vector $\mathbf{g}_m \in \mathbb{R}^{N}$. $\mathbf{g}_m[i] > 0$ if node $i$ is included in the $m$th node group, and $\mathbf{g}_m[i] = 0$ otherwise. 

We will assume the following about node's communication capabilities:

\begin{assumption}
    \label{assum:comms}
    Nodes in the same node group can communicate with any other node in the same group freely.
\end{assumption}

Recalling the motivating example, if node groups are determined by geographical proximity, it is reasonable for communication to be possible within them.

Let $\mathcal{T}$ represent the set of distinct tasks, and let $T$ represent the number of distinct tasks.
Every task $t$ comes with a vector of attributes $\mathbf{p}^t_i$ representing requirements for running the task on node $i$ (e.g., the energy required to complete the task; see Table \ref{tab:sciencejobs} in section \ref{sec:simulation} for a more in-depth example). Note that, because nodes may have different computing resources and efficiency, the task attributes 
may differ between nodes. 

The goal of the system is to find a scheduling policy $\mathbf{x} = \begin{bmatrix}\mathbf{x}_1 \\ \mathbf{x}_2 \\ \vdots \\ \mathbf{x}_N\end{bmatrix}$
such that for each node $i$ with associated quality-of-service (QoS) metric $U_i(\mathbf{x}_i)$, the total QoS across all nodes, $\sum_{i\in \mathcal{N}} U_i(\mathbf{x}_i)$, is maximized. Each $\mathbf{x}_i \in \mathbb{R}^{T}$ represents the allocation of node $i$'s resources to each task, and is bounded above by the hardware constraints of $i$, $\mathbf{b}_i$, and below by $0$, forming the domain $\mathcal{X}_i$. \hwa{Finally, let $\mathbf{e}_i \in \mathbb{R}^T$ be the unit vector with $\mathbf{e}_i[j] = 1$ if $i =j$, and 0 otherwise.}
\hwa{Lastly, we consider $(\mathbf{g}_m \otimes\mathbf{e}_t)^{\top}\mathbf{x}$, the total effort allocated by node group $m$ to its associated task $t$, in which $\otimes$ indicates the Kronecker product, and a QoS minimum that must be achieved across that node group, $q_m \in \mathbb{R}$.} We assume that $q_m$ is known to all nodes in the $m$th node group.

Taking all of this together yields the optimization problem,

\begin{subequations}
\label{eq:staticproblem}
\begin{equation}
    \max_{\mathbf{x}_i\in\mathcal{X}_i}
    \sum_{i \in \mathcal{N}} U_i(\mathbf{x}_i) 
    \label{eq:thebigproblem}
\end{equation}
\begin{equation}
        \hwa{\text{ s.t. } \hspace{0.25cm} (\mathbf{g}_m \otimes\mathbf{e}_t)^\top\mathbf{x} \geq q_m \hspace{0.25cm} \forall m \in \mathcal{M}.}
    \label{eq:thebigconstraints}
\end{equation}
\end{subequations}

To ensure that strong duality holds for our theoretical guarantees in Section \ref{sec:algs}, we will make two more assumptions:

\begin{assumption}
\label{assum:concavity}
   Each function $U_i$ is concave in its argument.
\end{assumption}

\begin{assumption}
\label{assum:interior}
    Each set $\mathcal{X}_i$ is convex and has a non-empty interior with at least one strictly feasible point.
\end{assumption}

$U_i$ can be designed by the owners of the system, so Assumption \ref{assum:concavity} can be guaranteed. So long as each node has nonzero $\mathbf{b}_i$ and the QoS minimum can be achieved without any $\mathbf{x}_i = \mathbf{b}_i$ (both practical assumptions), Assumption \ref{assum:interior} is also satisfied.

We can easily extend our formulation to the case when the system is dynamic, which we model by including a time index $k$. There, we handle the case when the network is dynamic by having $G^t$ and $\mathcal{N}$ change over time, the case when the nodes themselves are dynamic by having $\mathcal{X}_i$ change over time, and the case where the network's priorities are dynamic by having $\mathbf{q}^t$ change over time. This yields the related problem,

\begin{subequations}
\begin{equation}
    \max_{\mathbf{x}_i(k)\in\mathcal{X}_i(k)}
    \sum_{i \in \mathcal{N}(k)} U_i(\mathbf{x}_i(k)) 
    \label{eq:thebigproblem_dynamic}
\end{equation}
\begin{equation}
        \hwa{\text{ s.t. } \hspace{0.25cm} (\mathbf{g}_m(k) \otimes\mathbf{e}_t)^\top\mathbf{x}(k) \geq q_m(k)\hspace{0.25cm} \forall m \in \mathcal{M}.}
    \label{eq:thebigconstraints_dynamic}
\end{equation}
\end{subequations}

The two problems are almost identical, so to simplify notation, for the purposes of theoretical analysis we primarily consider the problem described by (\ref{eq:thebigproblem}) and (\ref{eq:thebigconstraints}), and discuss the generalization to (\ref{eq:thebigproblem_dynamic}) and (\ref{eq:thebigconstraints_dynamic}) afterwards.

\section{Algorithm}
\label{sec:algs}


In order to solve (1), we first write its Lagrange dual. Here, for compactness, we rewrite the linear inequality constraints (\ref{eq:thebigconstraints}) as a single constraint below, 
\begin{equation}
    G\mathbf{x} \geq \mathbf{q}\space,
\end{equation}
where 
$\mathbf{q} = \begin{bmatrix}\mathbf{q}^1 \\ \mathbf{q}^2 \\ \vdots \\ \mathbf{q}^M\end{bmatrix}$, and \hwa{$G$ is a matrix with the $m$-th row being $(\mathbf{g}_m \otimes\mathbf{e}_t)^\top$.} We also define $\mathcal{X}$ to be the feasible region for $\mathbf{x}$ jointly defined by each $\mathcal{X}_i$. This yields the dual problem, 

\begin{equation}
\label{eq:dualproblem}
     \min_{\lambda\geq0} \max_{\mathbf{x} \in \mathcal{X}} \sum_{i \in \mathcal{N}} U_i(\mathbf{x}_i)  + \lambda^\intercal(G\mathbf{x}-\mathbf{q}).
\end{equation}

If strong duality holds, then (\ref{eq:dualproblem}) will return a solution for $\mathbf{x}$ that also solves (1). By solving the dual problem instead of the primal problem, we are able to write an algorithm that can be implemented in a distributed way. 
Thus, we propose Algorithm \ref{alg:dualmethod}, a standard dual-descent method with step size $\alpha>0$. 


\begin{algorithm}
\caption{The Dual-Descent Method}
\label{alg:dualmethod}
    \begin{algorithmic}
        \State Initialize $\mathbf{x} = 0$, $\lambda = 0$, $k = 0$, $\alpha>0$.
        \For{$k \geq 0$}
        \State $\mathbf{x}(k+1) = \arg\max_{\mathbf{x}\in \mathcal{X}} \sum_{i \in \mathcal{N}} U_i(\mathbf{x}_i)+\mathbf{\lambda}(k)^\mathbf{\intercal} G\mathbf{x}$
        \State $\mathbf{\lambda}(k+1) = \max\{\lambda(k)-\alpha(G\mathbf{x}(k+1)-\mathbf{q}),0\}$
        \EndFor
    \end{algorithmic}
    
\end{algorithm}


Note that there are two time scales in the system: the time scale that concerns the iterations of Algorithm \ref{alg:dualmethod}, the time scale that concerns changes in the environment and therefore changes the optimization problem. We will assume that our algorithm can run and converge faster than the problem changes, which we justify via our simulation in Section \ref{sec:simulation}.

\subsection{Theoretical Guarantees}


We now proceed to show that dual-descent can be implemented in a distributed way within our communication constraints, and that it converges to the optimal solution. 

\begin{theorem}
    Algorithm \ref{alg:dualmethod} can be performed with the communication constraints laid out in Section \ref{sec:formulation}.
\end{theorem}

\begin{proof}

Note that, for the update of $\mathbf{x}$, $\lambda$ is constant, and each $U_i$ only depends on its corresponding $\mathbf{x}_i$. Therefore, each node can optimize and update their own $\mathbf{x}_i$ independently from each other, since all the algorithm needs is the $\arg\max$, not the actual maximum value. Node $i$ will have the correct $\mathbf{x}_i$, but possibly incorrect $\mathbf{x}_j$ for $j \neq i$. However, this behavior is not an issue, as explained below.

Concerning the update of $\lambda$, consider the update for a particular element of $\lambda$, $\lambda^t_m$. $\lambda^t_m$ is the dual variable associated with task $t$ and corresponding node group $m$. Using the structure of $G$, the update for $\lambda^t_m$ is
\begin{equation}
    \lambda^t_m(k+1) = \max\{\lambda^t_m(k)-\alpha( (\mathbf{g}_m \otimes \mathbf{e}_t)^\top\mathbf{x}(k+1) - q_m),0\}.
\end{equation}
Recalling that $q_m$ is constant and known to all nodes in its corresponding node group by assumption, it poses no issue. Using the fact that $\mathbf{g}_m[i] = 0$ if $i$ isn't in the $m$th node group for task $t$, it follows that $\lambda^t_m(k+1)$ only depends on $\mathbf{x}[i]$ for which $i$ is in the that node group. By Assumption \ref{assum:comms}, nodes in the same node group are able to freely communicate with each other, so the nodes in node group $m$ can coordinate to update $\lambda^t_m(k+1)$ correctly. 


Thus, both update steps conform to the communication constraints as desired.

\end{proof}

To prove convergence to the correct value, we will use Slater's condition \cite{Slater}:

\begin{theorem}[Slater's condition]
    \label{thm:slater}
    Given an optimization problem of the form
    \begin{equation}
    \begin{split}
        \min_{\mathbf{x}\in\mathcal{X}} \text{ } & f_0(\mathbf{x}) \\
        \text{s.t. } & f_i(\mathbf{x}) \leq 0, \hspace{0.25cm} i=1\dots n \\
         & A\mathbf{x} = \mathbf{b},
    \end{split}
    \end{equation}
with each $f_0, f_1, \dots f_n$ convex, if there exists an $\mathbf{x}$ in the interior of the convex set $\mathcal{X}$ such that $f_i(\mathbf{x})<0$ for all $i$, then strong duality holds.
\end{theorem}

\begin{proof}
    See \cite{Slater} or \cite{Boyd_Vandenberghe_2022}.
\end{proof}

\begin{theorem}
    \label{thm:convergence_static}
    Under the assumptions outlined in Section \ref{sec:formulation}, Algorithm \ref{alg:dualmethod} converges to the optimal solution, $\mathbf{x}^*$.
\end{theorem}

\begin{proof}
    We first rewrite (\ref{eq:staticproblem}) as an equivalent minimization problem:

    \begin{equation}
    \label{eq:staticproblem_rewritten}
    \begin{split}
       & \min_{\mathbf{x}_i\in\mathcal{X}_i}
     \sum_{i \in \mathcal{N}} -U_i(\mathbf{x}_i) \\
    \text{s.t. }  -(\mathbf{g}_m(k) & \otimes\mathbf{e}_t )^\top\mathbf{x}  \leq -q_m \quad \forall m \in \mathcal{M}.
    \end{split}
    \end{equation}
    
    Because each $U_i(\mathbf{x}_i)$ is concave and the constraints are linear, both the objective and the constraints of (\ref{eq:staticproblem_rewritten}) are convex. Combined with Assumptions \ref{assum:concavity} and \ref{assum:interior}, (\ref{eq:staticproblem_rewritten}) satisfies Slater's condition, so strong duality holds. 
    
    Because (\ref{eq:staticproblem_rewritten}) is convex, the dual problem is convex, and so Algorithm \ref{alg:dualmethod} will converge to the dual optimal $(\mathbf{x}^*,\lambda^*)$ since it is a variant of gradient descent. Therefore, since Algorithm \ref{alg:dualmethod} achieves the dual optimal and strong duality applies, $\mathbf{x}^*$ will be primal optimal so long as it is primal-feasible, which will be the case when a feasible solution exists for the primal problem \cite{Boyd_Vandenberghe_2022}.
\end{proof}

Given some additional assumptions on the feasibility of (\ref{eq:staticproblem}) over time, we can come up with a similar theorem for the dynamic case.

\begin{theorem}
    At each time step, if the assumptions in Section \ref{sec:formulation} hold and (2) is feasible, algorithm \ref{alg:dualmethod} converges to the optimal solution at that time step, $\mathbf{x}^*$.
\end{theorem}

\begin{proof}
    If the conditions hold at each time step, then the algorithm applied to the dynamic system is exactly equivalent to the static case, except with different initial conditions. Therefore Theorem \ref{thm:convergence_static} holds, and the algorithm returns the primal optimal as desired when run at each time step.
\end{proof}



\section{Numerical Simulation}
\label{sec:simulation}

To demonstrate the proposed modeling and algorithm, we conducted a simulation study motivated by the scientific application, wildfire detection and monitoring \cite{altintas2015towards}, and using real-world data from the Sage edge computing platform~\cite{Beckman:2019} for the application. We first describe the network conditions and properties of the tasks to be scheduled, and then present the simulated behavior of the network under two different ranges of conditions.

\subsection{Data Preparation with Sage}

To design a simulation that reflects the real setting of the motivating example in Section \ref{sec:theexample} (see Fig. \ref{fig:wildfire}), we use Sage's cyberinfrastructure and generate a dataset. The dataset includes the user requirements (i.e., the minimum QoS) for the tasks and total energy and memory-time required to run the AI task on the Sage node. While we can compute the memory-time for a task easily by multiplying the task's runtime by its memory requirement, measuring the energy requirement is a bit more difficult. We first gather the performance of the AI tasks currently scheduled on Sage nodes and system power consumption during the task execution on the node using the Sage application programming interface (API). Then, we subtract the Sage node's idle power consumption from the system power to estimate the total energy used only for the task. The simulation thus uses the energy and memory consumption of the real AI tasks running on different Sage edge devices.

As in the motivating example (Section \ref{sec:theexample}), there are two tasks that we examined: task 1, which monitors potential wildfires for early warning and detection, and task 2, which analyzes and tracks the movement of clouds. The hardware requirements of the tasks are shown in Table \ref{tab:sciencejobs}, with two different measurements depending on what type of node ran the task.

\begin{table*}[htbp]
\caption{Characteristics of the target science tasks.}
\begin{center}
\begin{tabular}{|p{0.13\linewidth}|p{0.15\linewidth}|p{0.1\linewidth}|p{0.09\linewidth}|p{0.09\linewidth}|p{0.11\linewidth}|p{0.11\linewidth}|}
\hline
\multirow{2}{*}{\textbf{Task Name}} & \multirow{2}{*}{\textbf{Arrival Rate}} & 
\multirow{2}{*}{\begin{tabular}{@{}c@{}} \textbf{Min Runs per} \\ \textbf{Node per Cycle} \end{tabular}} &\multicolumn{2}{c|}{\textbf{Energy Consumption (Joules)}} & \multicolumn{2}{c|}{\textbf{Memory Consumption (GB seconds)}} \\
& & & Nvidia Jetson & Raspberry Pi & Nvidia Jetson & Raspberry Pi  \\
\hline
WildFire Monitoring~\cite{dewangan2022figlib} & Every 5 minutes & 2 & 68  & 42 & 50.4 &  67.2 \\
\hline
Cloud Analysis~\cite{park2021prediction} & Every 10 minutes & 1 & 70 & 48 & 240 &  277.5 \\
\hline
\end{tabular}
\vspace{-15pt}
\label{tab:sciencejobs}
\end{center}
\end{table*}

\subsection{Simulation Structure}


\begin{figure}
    \centering
    \includegraphics[width=\linewidth]{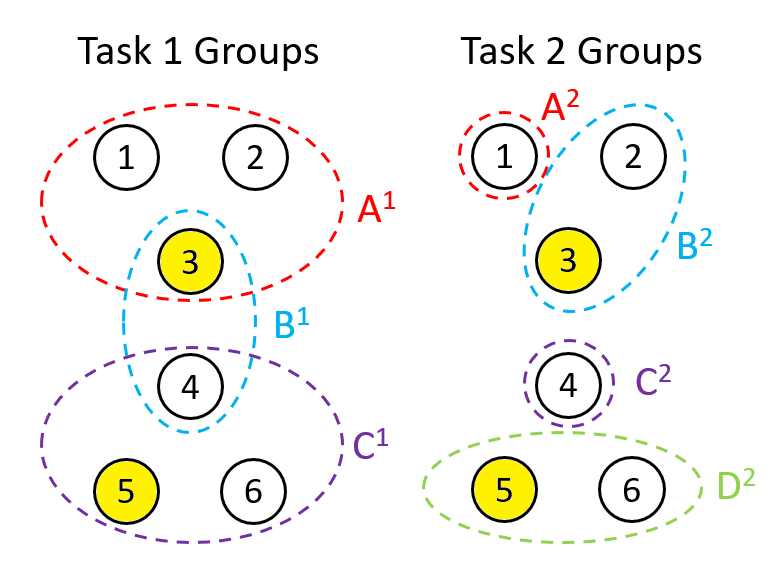}
    \caption{The network used in our simulation. The wildfire monitoring task (task 1) uses the node groups on the left, while the cloud analysis task (task 2) uses the node groups on the right. Nodes 3 and 5, indicated in yellow, are higher-performance Nvidia Jetson devices, while nodes 1, 2, 4, and 6, indicated in white, are lower-end Raspberry Pi devices.}
    \label{fig:simnetwork}
\end{figure}

The network we simulated is shown in Fig. \ref{fig:simnetwork}. There are 6 nodes, and two sets of node groups, each corresponding to a different task. Nodes 3 and 5, indicated in yellow, represent edge computing nodes that are equipped with a Nvidia Jetson device: these are higher performance, higher power nodes, compared to nodes 1, 2, 4, and 6, which are the nodes with a Raspberry Pi for computing: these are lower performance, lower power nodes. The hardware capabilities of each node that are relevant to the simulation are shown in Table \ref{tab:nodecapabilities}. Because we have two types of nodes with different capabilities, tasks run on one type will have different hardware requirements than the tasks run on the other. Additionally, tasks run on the higher performance nodes are considered more valuable in terms of QoS, since a higher performance node may complete the task more quickly, or may produce higher quality outputs than a lower performance node would.

\begin{table}[]
    \centering
    \caption{The hardware constraints of the nodes in our simulation.}
    \begin{tabular}{|l|c|c|c|c|}
    \hline
        \textbf{Node Type} & \begin{tabular}{@{}c@{}} \textbf{Available} \\ \textbf{Power} \\ \textbf{(W)} \end{tabular} & \begin{tabular}{@{}c@{}} \textbf{Energy} \\ \textbf{per Cycle} \\ \textbf{(J)} \end{tabular} & \begin{tabular}{@{}c@{}}\textbf{Available} \\ \textbf{Memory} \\ \textbf{ (GB)} \end{tabular} & \begin{tabular}{@{}c@{}} \textbf{Memory-time} \\ \textbf{per Cycle} \\ \textbf{(GBsec)} \end{tabular} \\ \hline
        Nvidia Jetson & 10  &6000& 6 &3600   \\ \hline
        Raspberry Pi & 4 &2400 & 3 & 1800 \\ \hline
    \end{tabular} 
    \label{tab:nodecapabilities}
\end{table}

The simulation is structured around assigning tasks on a 10-minute cycle. During that 10-minute cycle, it is assumed that each node is capable of handling its own internal schedule to complete its assigned tasks. So long as the total energy of tasks assigned in that cycle doesn't exceed the energy available, the node will complete those tasks, Similarly, we calculate the total memory-time needed to complete the task, and ensure that the total memory-time of assigned tasks doesn't exceed the maximum memory-time of the node per cycle. Based on the arrival rate of each task, we can determine the minimum number of times the task needs to be run per node, per cycle. 

With all the above, we can now present the specific parameter choices for (\ref{eq:staticproblem}), the optimization problem the scheduler will solve. The policy vector $\mathbf{x}$ represents the number of copies of each task that are run on each node during the cycle. The domain of optimization for each node $i$, $\mathcal{X}_i$, is lower bounded by 0, and upper bounded by the two linear inequalities, 
\begin{subequations}
    \begin{equation}
        p_i^1x_i^1 + p_i^2x_i^2 \leq b_i[pow],\hspace{1cm} \text{(Energy)}
    \end{equation}
    \begin{equation}
        d_i^1x_i^1 + d_i^2x_i^2 \leq b_i[mem], \hspace{1cm} \text{(Memory)}
    \end{equation}
\end{subequations}
where $p_i^t$ is the average energy in $J$ required to run task $t$ on node $i$, $d_i^t$ is the total memory in GB seconds required to run $t$ on $i$, $b_i[pow]$ is the maximum energy per cycle for $i$, and $b_i[mem]$ is the maximum memory-time per cycle for $i$. 

The node group vectors, $\mathbf{g}_m$, are vectors of 1s and 0s, with $\mathbf{g}_m^t[i] = 1$ if node $i$ is in node group $m$, and $\mathbf{g}_m^t[i] = 0$ otherwise. The QoS minimums are defined as $[q_{A1} \space q_{B1} \space q_{C1} ] = [\text{6 4 6}]^\intercal$ and $[q_{A2} \space q_{B2} \space q_{C2} \space q_{D2}] = [\text{1 2 1 2}]^\intercal$ in accordance with each node group having to run the minimum number of tasks per node per cycle, as in Table \ref{tab:sciencejobs}. Lastly, we use a linear objective function with the following constants:

\begin{equation*}
\label{eq:sim_objective}
    \hspace{0.5cm} U_i(\mathbf{x}_i) = \bigg\{ 
    \begin{split}
        \hspace{-3cm} 1.3x_i^1  +6x_i^2 \hspace{0.5cm} & \text{$i$ is high power} \\ 
        \hspace{-3cm} x_i^1  + 4x_i^2 \hspace{0.5cm} & \text{$i$ is low power} 
    \end{split}
    \bigg\}. \hspace{1.15cm} \text{(9)} \hspace{-1.15cm}
\end{equation*}

This represents the case where task 2 is considered 4 times as valuable as the wildfire task (to account for the increased power and memory requirements), and task 1 / task 2 being run on higher power nodes is considered 1.3 / 1.5 times more valuable than if it were run on lower power nodes, respectively.

\subsection{Simulation Runs and Results}

The first simulation we conducted examined the system's performance when its node's maximum capabilities were degraded. Specifically, we lowered the maximum energy and memory that nodes 3 and 6 could allocate to tasks, doing so in 1\% increments starting from 0 degradation (full capability) all the way to 1 degradation (no capability). A situation of this sort might occur if a node is in an environment in which it needs to conserve power, so it artificially lowers its own capabilities.

To emulate our algorithm's ability to respond to emergent situations, we ran our algorithm and calculated the resulting QoS assuming it has perfect knowledge of the node's degradation. The degradation is captured by updating $\mathcal{X}$ and letting the algorithm run until convergence before updating the degradation again. We are able to reasonably emulate our algorithm's dynamicity in this way because it converged in seconds, in comparison to the cycle length of 10 minutes. We call the algorithm running under these circumstances the \textit{dynamic policy}. 

For comparison, we also considered a \textit{static policy} that uses the result of the optimization problem with 0 degradation, with proportional scaling down for the policies associated with degraded nodes -- e.g., with 0.1 degradation on node $i$, and an original policy of $\mathbf{x}_i$, the static policy would assign a new policy of $0.9\mathbf{x}_i$, regardless of QoS constraint violations. The static policy thus emulates the case in which nodes cannot coordinate at the edge and respond to sudden changes.

The second simulation we ran examined how the scheduler handled changes to the QoS minimum. While holding the QoS minimum to be the same for task 2, we varied the per-node QoS minimum for the task 1 between 2 runs to 16 runs, emulating an increased demand for wildfire detection caused by environmental factors that would lead to an increased chance for a wildfire to start. In this way, we examined the effects of changing the upper bounds of $\mathbf{x}$ in the first simulation, and changing the lower bounds in the second.

\begin{figure}
    \centering
    \includegraphics[width=\linewidth]{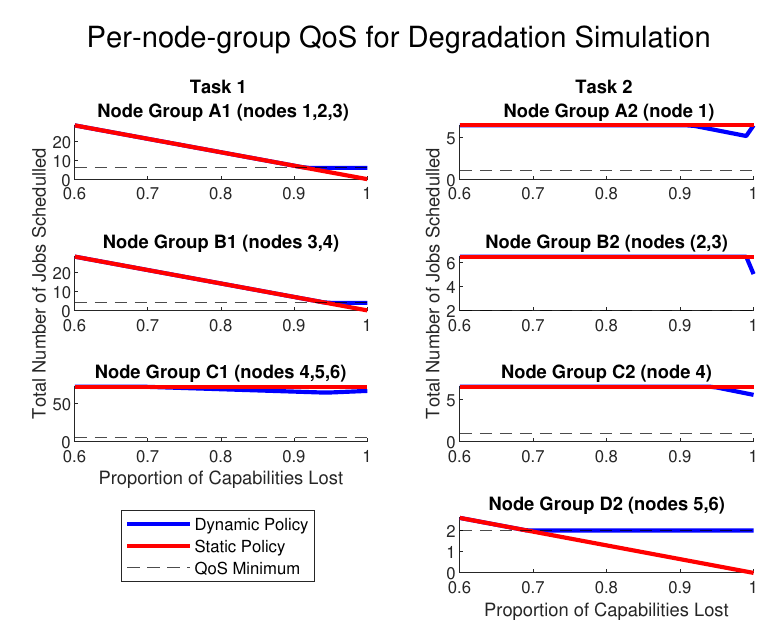}
    \caption{The number of tasks ran on each node group for the degradation simulation. The left shows the node groups for task 1, while the right shows the node groups for task 2. The dynamic and static policy overlay each other at first, but the static policy fails to uphold the QoS minimum in high degradation while the dynamic policy maintains it.}
    \label{fig:degrade group qos}
\end{figure}

The results are shown in Figures \ref{fig:degrade group qos}-\ref{fig:qos node qos}. Figure \ref{fig:degrade group qos} shows the total number of jobs ran on each node group for the degradation simulation, with task 1 on the left and task 2 on the right. In high degradation ($>0.6$), the static policy begins to fail while the dynamic policy effectively reassigns tasks to meet the QoS minimum. We omitted the analogous plot for the simulation that changes the QoS minimum because it behaves similarly: the static policy fails while the dynamic policy continues to return a feasible solution.

\begin{figure}
    \centering
    \includegraphics[width=\linewidth]{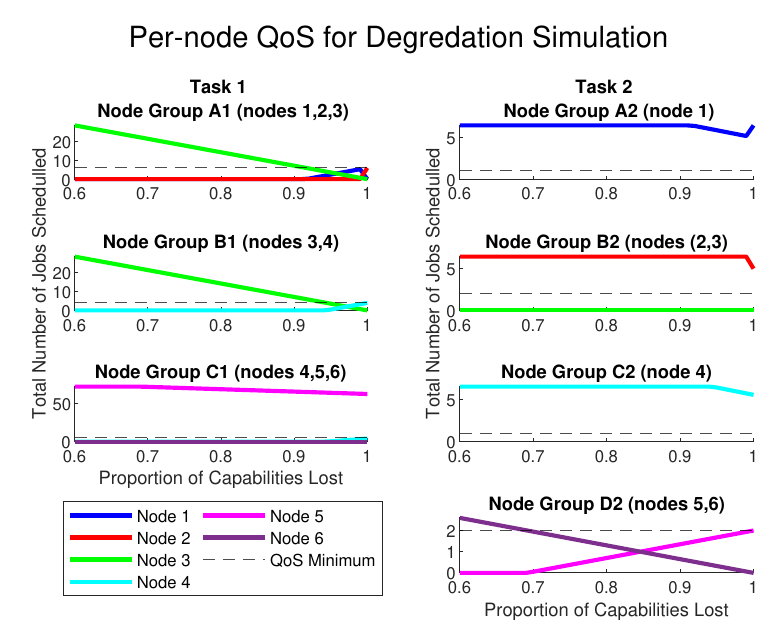}
    \caption{The number of tasks ran on each node for the degradation simulation. As nodes 3 and 6 start falling below the minimum QoS requirements, other nodes in their node groups reallocate resources to fill the gaps.}
    \label{fig:degrade node qos}
\end{figure}

To see what is happening more clearly, Fig. \ref{fig:degrade node qos} shows the tasks run by each node for the degradation simulation. When degradation reaches 0.7, node 6, which meets the QoS minimum for task 2 for node group D2, begins to fall below the QoS minimum. In response, node 5 begins working on task 2 in order for the node group to meet the QoS minimum, resulting in a reduction in its effort towards task 1. Similarly, when degradation reaches 0.9, node 3 begins to fall below the minimum for task 1, resulting in nodes 1 and 4 filling the gap. The sudden change in nodes 1 and 2 at full degradation is due to there being multiple solutions for the optimization problem in our setup, since 1 and 2 have identical capabilities and objective functions. Because the solver must set $\mathbf{x}_3$ and $\mathbf{x}_6$ to $0$ at full degradation, its search for the optimal solution behaves a little differently.

\begin{figure}
    \centering
    \includegraphics[width=\linewidth]{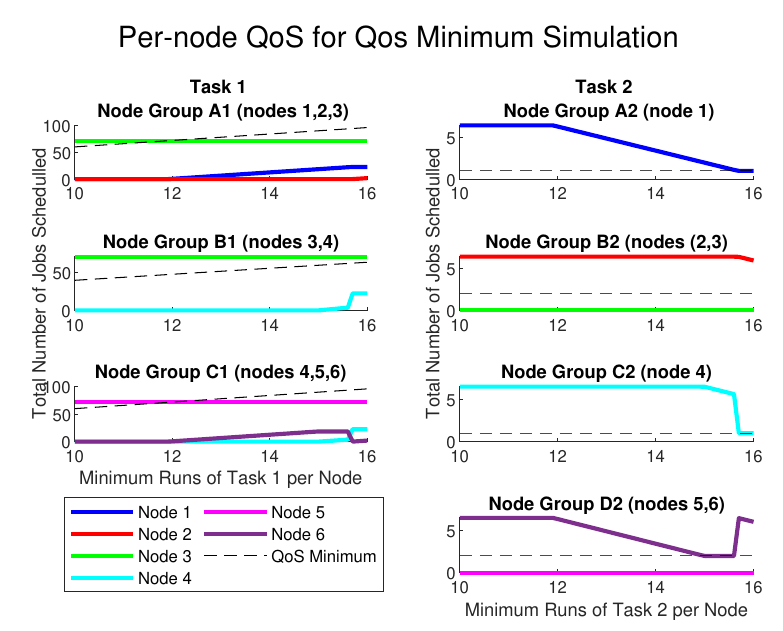}
    \caption{The number of tasks ran on each node for the QoS minimum simulation. As the QoS minimum for task 1 increases, more nodes shift their resources to meet the higher requirements.}
    \label{fig:qos node qos}
\end{figure}

Finally, Fig. \ref{fig:qos node qos} shows the same information as Fig. \ref{fig:degrade node qos}, but for the simulation that changes the QoS minimum. When the minimum runs per node exceeds 12, nodes 1 and 6 begin to switch over from task 2 to task 1 to aid node groups A1 and C1 (B1 has fewer nodes, and thus a lower QoS minimum). Then, when the minimum runs per node exceeds 15, nodes 1 and 6 cannot devote any more effort because of their QoS constraints concerning task 2, so node 4 switches over. As before, because nodes 4 and 6 have identical capabilities and objectives concerning task 1, there are multiple optimal solutions, and a sudden switch between them occurs once node 2 begins helping with task 1.

\section*{Conclusion}

We present a novel framework for task allocation in edge computing systems using an optimization problem. The framework is flexible enough to capture heterogeneity in the network's nodes and different characteristics and requirements in the network's assigned tasks. We propose a distributed algorithm to solve our optimization problem, and demonstrate its ability to respond to a dynamic environment through a simulation with real world data. Future work can take many different directions, including verification of the framework and proposed algorithm on actual hardware, expanding the flexibility of the framework by using a more general objective function, and studies comparing our optimization framework to the AI methods for task allocation that are increasing in popularity.


\bibliographystyle{IEEEtran}
\bibliography{bib}

\end{document}